\documentclass[11pt]{article}
\usepackage{anysize}\marginsize{3.5cm}{3.5cm}{1.3cm}{2cm}

\makeatletter\@ifundefined{pdfpagewidth}{}{\pdfpagewidth=21.0cm\pdfpageheight=29.7cm}\makeatother % als Information für pdfTeX
\usepackage{amssymb}
\headsep=\baselineskip
\makeatletter

\let\orig@item=\@item \def\@item[#1]{\orig@item[\rm #1]}
\renewenvironment{abstract}{\begin{quote}\footnotesize\textbf{\abstractname.}}{\end{quote}\bigskip}
\renewcommand\@seccntformat[1]{\csname the#1\endcsname.\enspace}
\renewcommand\paragraph{\@startsection{paragraph}{4}{\z@}{1\baselineskip}{-0.5em}{\normalsize\bfseries}}
\let\origcaption=\caption \renewcommand\caption[1]{\parbox{0.8\textwidth}{\origcaption{#1}}}

\renewcommand\@begintheorem[2]{\trivlist\item[\hskip\labelsep{\bfseries#1 #2.}]\it}
\renewcommand\@opargbegintheorem[3]{\trivlist\item[\hskip\labelsep{\bfseries#1 #2}] {\bfseries(#3).}\enspace\it\ignorespaces}
\makeatother

\newtheorem{satz}{Satz}[section]
\makeatletter \newcommand\numberwithin{\@addtoreset} \numberwithin{equation}{satz} \makeatother

\newtheorem{theorem}[satz]{Theorem}

\newtheorem{lemma}[satz]{Lemma}
\newtheorem{proposition}[satz]{Proposition}
\newtheorem{remark}[satz]{Remark}
\newtheorem{introtheorem}{Theorem}
\newenvironment{proof}[1][Proof]{\trivlist\item[\hskip\labelsep{\it #1.}]}{\hspace*{\fill}$\Box$\endtrivlist}

\renewcommand\emptyset{\varnothing}  % from amssymb
\renewcommand\ge{\geqslant}  % from amssymb
\renewcommand\le{\leqslant}  % from amssymb
\renewcommand\geq{\geqslant}  % from amssymb
\renewcommand\epsilon{\varepsilon}
\renewcommand\phi{\varphi}
\renewcommand\bar{\overline}

\renewcommand\O{{\cal O}}
\renewcommand\P{\mathbb P}
\newcommand\R{\mathbb R}
\newcommand\C{\mathbb C}
\newcommand\Q{\mathbb Q}
\newcommand\Z{\mathbb Z}

\newcommand\isomto{\mathop{\longrightarrow}\limits^\sim}
\newcommand\bilin[2]{\left\langle#1,#2\right\rangle}
\newcommand\vecspan[1]{\left\langle#1\right\rangle}
\newcommand\longunion[3]{#1_{#2}\cup\dots\cup#1_{#3}}
\newcommand\tab[2][t]{\begin{tabular}[#1]{@{}l@{}}#2\end{tabular}}
\newcommand\liste[3]{\mbox{$#1_{#2},\dots,#1_{#3}$}}
\newcommand\set[1]{\left\{#1\right\}}
\newcommand\sset[1]{\left\{\,#1\,\right\}}
\newcommand\with{\ \vrule\ }
\newcommand\be{\begingroup\arraycolsep=0.13888em\begin{eqnarray*}}
\newcommand\ee{\end{eqnarray*}\endgroup}
\newcommand\eqnref[1]{(\ref{#1})}
\newcommand\eqdef{\stackrel{\rm def}{=}}
\newcommand\tensor{\otimes}

\newlength\matrcolsep \matrcolsep=\arraycolsep
\newcommand\matr[2][*{\maxmatrcols}{c}]{{\arraycolsep=\matrcolsep\left(\begin{array}{#1}#2\end{array}\right)}}
\newcommand\inverse{^{-1}}

\newcommand\interior[1]{\stackrel{\circ}{#1}}
\newcommand\compact{\itemsep=0cm \parskip=0cm}
\newcommand\newop[2]{\newcommand#1{\mathop{\rm #2}\nolimits}}
\newop\mult{mult} % Multiplizität
\newop\NS{NS} % N\'eron-Severi-Gruppe
\newop\Neg{Neg}
\newop\BigCone{Big} % \Big gibt es schon für Klammergrößen!
\newop\vol{vol}
\newop\Nef{Nef}
\newop\Null{Null}
\newop\Bl{Bl}

\def\dyndiag#1 #2 #3{\unitlength1cm\begin{picture}(#1,#2)(0,-0.1)#3\end{picture}}

\newcommand{\circlediam}{0.15}
\def\punkt #1 #2 {\put(#1,#2){\circle*{\circlediam}}}
\def\pu #1{\punkt {#1} 0 }

\def\linie #1 #2 #3 #4 {\put(#1,#2){\line(#3,#4){1}}}
\def\li #1 {\linie {#1} 0 1 0 }
\def\dotli #1{\multiput(#1,0)(0.1,0){10}{\circle*{0.05}}}
\def\twodotli #1{\multiput(#1,0)(0.1,0){20}{\circle*{0.05}}}
\def\ddotli #1 #2 {\multiput(#1,#2)(0.07,0.07){14}{\circle*{0.05}}}
\def\idotli #1 #2 {\multiput(#1,#2)(0.07,-0.07){14}{\circle*{0.05}}}

\def\numb #1 #2 #3 {\put(#1,#2){\raisebox{0.15cm}{\makebox[0cm]{\footnotesize #3}}}}
\def\unumb #1 #2 #3 {\put(#1,#2){\raisebox{-0.4cm}{\makebox[0cm]{\footnotesize #3}}}}
\def\lnumb #1 #2 #3 {\put(#1,#2){\raisebox{-0.1cm}{\makebox[0cm][r]{\footnotesize #3\ }}}}
\def\rnumb #1 #2 #3 {\put(#1,#2){\raisebox{-0.1cm}{\makebox[0cm][l]{\ \footnotesize #3}}}}

\newcommand\ADEdiagrams{
   $A_n$
   \dyndiag 7 1 {%
      \pu1
      \pu2
      \pu3
      \pu4
      \pu6
      \li1
      \li2
      \li3
      \twodotli4
      }
   ($n\ge 1$ vertices)

   $D_n$
   \dyndiag 7 1 {%
      \punkt 1 0.5
      \punkt 1 -0.5
      \pu2
      \pu3
      \pu4
      \pu6
      \linie 1 0.5 2 -1
      \linie 1 -0.5 2 1
      \li2
      \li3
      \twodotli4
      }
   ($n\ge 4$ vertices)

   $E_n$
   \dyndiag 7 1 {%
      \pu1
      \pu2
      \pu3
      \pu4
      \pu5
      \pu6
      \punkt 3 -1
      \li1
      \li2
      \li3
      \li4
      \linie 3 0 0 -1
      \dotli5
      }
   ($n=6$, $7$ or $8$ vertices)
   }

\newcommand\FigureQuarticChambers{
   \unitlength 1mm % = 2.845pt
   \linethickness{0.4pt}
   \ifx\plotpoint\undefined\newsavebox{\plotpoint}\fi % GNUPLOT compatibility
   \begin{picture}(133,51)(0,0)
   \put(59,4){\line(-3,5){27}}
   \put(5,4){\line(1,0){54}}
   \put(22,32){\line(1,0){20}}
   \put(42,32){\line(-5,-6){10}}
   \put(32,20){\line(-5,6){10}}
   \put(76,4){\line(3,5){27}}
   \put(130,4){\line(-3,5){27}}
   \put(76,4){\line(1,0){54}}
   \put(93,32){\line(1,0){20}}
   \put(113,32){\line(-5,-6){10}}
   \put(103,20){\line(-5,6){10}}
   \put(32,20){\line(-5,-6){13.333}}
   \put(32,20){\line(5,-6){13.333}}
   \put(3,2){\makebox(0,0)[cc]{$L_1$}}
   \put(62,2){\makebox(0,0)[cc]{$L_2$}}
   \put(74,2){\makebox(0,0)[cc]{$L_1$}}
   \put(133,2){\makebox(0,0)[cc]{$L_2$}}
   \put(32,51){\makebox(0,0)[cc]{$C$}}
   \put(103,51){\makebox(0,0)[cc]{$C$}}
   \multiput(103,20)(-.0568421053,-.0336842105){475}{\line(-1,0){.0568421053}}
   \multiput(103,20)(.0568421053,-.0336842105){475}{\line(1,0){.0568421053}}
   \put(5,4){\line(3,5){27}}
   \put(32,37){\makebox(0,0)[cc]{$C$}}
   \put(103,37){\makebox(0,0)[cc]{$C$}}
   \put(20,18){\makebox(0,0)[cc]{$L_1$}}
   \put(91,18){\makebox(0,0)[cc]{$L_1$}}
   \put(43,18){\makebox(0,0)[cc]{$L_2$}}
   \put(114,18){\makebox(0,0)[cc]{$L_2$}}
   \put(32,9){\makebox(0,0)[cc]{$L_1$,\ $L_2$}}
   \put(103,9){\makebox(0,0)[cc]{$L_1+L_2$}}
   \end{picture}
   }

\begin{document}

\title{Weyl and Zariski chambers on K3 surfaces}
\author{Thomas Bauer and Michael Funke}
\date{June 2, 2010}
\maketitle
\thispagestyle{empty}

\begin{abstract}
   The big cone of every K3 surface admits two natural chamber
   decompositions: the decomposition into Zariski chambers, and
   the decomposition into simple Weyl chambers. In the present
   paper we compare these two decompositions and we study their
   mutual relationship: First, we give a numerical criterion for
   the two decompositions to coincide. Secondly, we study the
   mutual inclusions of Zariski and simple Weyl chambers.
   Finally, we establish the fact that -- even though the
   decompositions themselves may differ -- the number of Zariski
   chambers always equals the number of simple Weyl chambers.
\end{abstract}

\section*{Introduction}

   The purpose of this note is to study two natural chamber
   decompositions of the big cone on K3 surfaces: On the one hand
   one has the decomposition into Zariski chambers, which is
   important when considering base loci and volumes of line
   bundles (see \cite{BKS}), and on the other hand one has the
   well-known decomposition into simple Weyl chambers, which is
   given by the hyperplanes that are dual to the negative curves.
   We compare these two decompositions and we study their
   geometry -- by determining the mutual inclusions of chambers
   and by asking for the number of chambers into which the big
   cone is divided.

   For a more detailed description, consider a smooth projective
   variety $X$ of dimension $n$
   and a divisor $D$ on $X$. One defines the
   \textit{volume} of $D$ as the real number
   $$
      {\rm vol}_X(D)\stackrel{\rm def}{=}\limsup_k{\frac{h^0(X,kD)}{k^n/n!}}\ , \
   $$
   and one calls $D$ \textit{big} if its volume is positive. The
   big cone $\BigCone(X)$ is then the cone in the N\'eron-Severi
   vector space $\NS(X)\tensor\R$ that is generated by the big
   divisors. These have recently attracted a lot of attention, as
   it turned out that many geometric, cohomological, and
   numerical aspects of the picture that one has classically for
   ample line bundles extend to big divisors (see
   \cite{nak1,PAG,elmnp1,elmnp}). The main result of \cite{BKS}
   states that on surfaces the big cone admits a locally finite decomposition
   into rational locally polyhedral subcones such that
   \begin{itemize}\compact
   \item
      on each subcone the volume
      function is given by a single polynomial of degree two, and
   \item
      in the interior of each of the subcones the stable base
      loci are constant.
   \end{itemize}
   Both facts are explained by the variation of the Zariski
   decomposition of big divisors: In the interior of the subcones
   the support of the negative part of the Zariski decomposition
   is constant. These subcones have therefore been called
   \emph{Zariski chambers} in~\cite{BKS}.

   Suppose now that $X$ is a K3 surface. In that case,
   an interesting second point of view is provided by the
   \textit{Weyl chamber decomposition}. We consider here in
   particular \textit{simple} Weyl chambers,
   which are defined by the intersection
   behaviour of big divisors with $(-2)$-curves
   (see Sect.~\ref{sect-chambers} for details).

   Given that there are two natural decompositions of the big
   cone of a K3 surface, it is an obvious task to compare them.
   In \cite[Sect.~3.2]{BKS} the erroneous claim was made that the
   two decompositions always coincide. They may however differ,
   depending on the geometry of the surface, and our first result
   gives the precise condition when this is the case:

\begin{introtheorem}\label{introthm-equality-criterion}
   Let $X$ be a K3 surface. Then the following conditions are
   equivalent.
   \begin{itemize}\compact
   \item[(i)]
      The interiors of the Zariski chambers coincide with the
      simple Weyl chambers.
   \item[(ii)]
      There is no pair of
      $(-2)$-curves $C_1,C_2\subset X$ such that
      $$
         C_1\cdot C_2= 1 \,.
      $$
   \end{itemize}
   There exist K3 surfaces such that (i) and (ii) hold, and
   there exist K3 surfaces where these conditions do not hold.
\end{introtheorem}

   This theorem is a consequence of the following more general
   result about the geometry of Zariski and Weyl chambers. To
   state it, we will use the following notation: When $\mathcal
   S=\set{\liste C1r}$ is a set of $(-2)$-curves whose
   intersection matrix is negative definite, then $Z_{\mathcal
   S}$ will denote the Zariski chamber consisting of the big divisors
   whose negative part is supported on $\longunion C1r$,
   and $W_{\mathcal S}$ will
   denote the simple Weyl chamber consisting of the divisors that
   intersect $\liste C1r$ negatively and all other $(-2)$-curves
   positively (see Sect.~\ref{sect-chambers} for the details).
   Furthermore, we will write $\interior Z_S$ to denote the
   interior of the chamber $Z_S$.

\begin{introtheorem}\label{introthm-inclusions}
   Let $X$ be a K3 surface, and let $\mathcal S$ be a set of
   $(-2)$-curves on $X$ whose intersection matrix is negative
   definite.
   \begin{itemize}\compact
   \item[(i)]
      We have $W_{\mathcal S}\subset Z_{\mathcal S}$ if and only
      if the following condition holds: If $C'$ is a curve with
      $C'\notin\mathcal S$ such that the intersection matrix of
      the set $\mathcal S\cup\set{C'}$ is
      negative definite, then $C'\cdot C=0$ for all $C\in\mathcal
      S$.
   \item[(ii)]
      We have $\interior Z_{\mathcal S}\subset W_{\mathcal S}$ if and only
      if $C_1\cdot C_2\ne 1$ for all curves $C_1,C_2\subset S$.
   \end{itemize}
\end{introtheorem}

   We will study in Sect.~\ref{sect-example} in detail examples
   of K3 surfaces where various cases of inclusions and
   equalities of chambers occur.

   Given the fact that the decompositions may well be different,
   it is nice and somewhat surprising that on any K3 surface
   there are precisely as many Zariski chambers as there are Weyl
   chambers:

\begin{introtheorem}
   Let $X$ be a K3 surface. Then there is a canonical bijection
   between the set of Zariski chambers in $\BigCone(X)$ and the
   set of simple Weyl chambers in $\BigCone(X)$.
\end{introtheorem}

\section{Chamber decompositions}\label{sect-chambers}

\paragraph{Zariski chambers.}
   We recall very briefly the Zariski chamber decomposition from~\cite{BKS}.
   Let $X$ be a smooth projective surface over $\mathbb C$. To any big and nef
   $\R$-divisor $P$, one associates the \textit{Zariski
   chamber} $\Sigma_P$, which by definition
   consists of all divisors in
   $\BigCone(X)$ such
   that the irreducible curves
   in the negative part of
   the Zariski decomposition of $D$ are precisely the curves $C$
   with
   $P\cdot C=0$.
   One has by \cite[Lemma~1.6]{BKS}:

\begin{quote}\it
   For two big and nef divisors
   $P$ and $P'$, the Zariski
   chambers $\Sigma_P$ and $\Sigma_{P'}$ are
   either equal or disjoint.
   The Zariski chambers yield a decomposition of the
   big cone.
\end{quote}
   If $H$ is an ample divisor, then the interior of the
   chamber $\Sigma_H$ is the
   ample cone, its closure is the nef cone.
   (Note that $\Sigma_H$ itself need not be open or closed.)
   By way of abbreviation,
   we will refer to this chamber as the \emph{nef chamber} in the
   sequel.
   We will make frequent use of the following basic
   observation.

\begin{lemma}\label{lemma-neg-def}
   The set of Zariski chambers on a smooth projective surface
   $X$ that are different from the
   nef chamber is in bijective correspondence with the set of
   reduced divisors on $X$ whose intersection matrix is negative
   definite.
\end{lemma}

\begin{proof}
   The statement is Proposition~1.1 from \cite{BFN}. For the
   benefit of the reader, we briefly give the argument.
   For a chamber $\Sigma_P$ we consider the irreducible
   curves $\liste C1r$ with $P\cdot C_i=0$. Then the divisor
   $C_1+\dots+C_r$ has negative definite intersection matrix
   thanks to the index theorem.
   Conversely, given a reduced divisor
   $C_1+\dots+C_r$ with negative definite intersection matrix,
   we consider
   the divisor
   $$
      D \eqdef H+k(C_1+\dots+C_r) \,,
   $$
   where $H$ is a fixed ample divisor and $k$ a positive integer.
   This divisor is big, and for $k\gg0$ the negative part of
   its Zariski decomposition will have
   $C_1\cup\dots\cup C_r$ as its support.
   (The latter fact can for instance
   be seen from the computation of the
   Zariski decomposition according to \cite{Bau}.)
\end{proof}

   It will be useful to introduce two abbreviations: If $D$ is a
   big divisor on a surface, then $\Null(D)$ will denote the set
   of irreducible curves $C$ such that $D\cdot C=0$. If
   $D=P_D+N_D$
   is the Zariski decomposition of $D$, with nef part $P_D$ and
   negative part $N_D$, then $\Neg(D)$ will
   denote the set of components of $N_D$.
   In this notation, the Zariski chamber associated with a big
   and nef divisor $P$ is
   by definition
   $$
      \Sigma_P=\sset{D\in\BigCone(X)\with\Neg(D)=\Null(P_D)} \,.
   $$

\paragraph{Weyl chambers.}
   Let now $X$ be a K3 surface. Apart from the Zariski chamber
   decomposition, there is a second natural decomposition of the
   big cone: the decomposition into Weyl chambers. Consider the
   set $\mathcal R(X)$ of irreducible $(-2)$-curves (smooth
   rational curves of self-intersection $-2$) on $X$, also referred to as
   the set of \textit{simple roots} on $X$. Via the intersection
   product, each of the curves $C\in\mathcal R(X)$ defines a
   hyperplane $C^\perp=\set{D\with D\cdot C=0}$ in $\NS_\R(X)$.
   The connected components of the complement
   $$
      \BigCone(X)\setminus\bigcup_{C\in\mathcal R(X)}C^\perp
   $$
   yield a decomposition of (a dense open subset of)
   the big cone into subcones, the
   \textit{simple Weyl chambers}. The term \textit{simple}
   relates here to the fact that classically one considers not
   only simple roots, i.e, irreducible $(-2)$-curves, but
   \textit{roots}, i.e., $(-2)$-classes; taking chambers with
   respect to $(-2)$-classes leads in general to a finer
   decomposition, whose sets are called \textit{Weyl chambers}.
   Also note that classically one considers instead of the big
   cone $\BigCone(X)$ the \textit{positive cone} $C^+(X)$, i.e.,
   the cone of divisors $D$ with $D^2>0$ and $D\cdot H>0$ for
   some ample $H$. This cone, which is a subcone of the big cone,
   enjoys the advantage of being invariant under the Weyl group.
   In our situation, however, when comparing with the Zariski
   chamber decomposition, it is more natural to work in the big
   cone.

\paragraph{Comparing decompositions.}

   Our first result compares the two decompositions of
   $\BigCone(X)$.
   Proposition 3.9 of \cite{BKS} asserts that the Zariski chamber
   decomposition coincides on K3 surfaces with the Weyl chamber
   decomposition, but its proof is erroneous.
   Instead, the two decompositions may differ, and
   the following result gives the precise condition when this
   happens.

\begin{theorem}\label{thm-equality-criterion}
   Let $X$ be a K3 surface. Then the following conditions are
   equivalent.
   \begin{itemize}\compact
   \item[(i)]
      The interiors of the Zariski chambers coincide with the
      simple Weyl chambers.
   \item[(ii)]
      There is no pair of
      $(-2)$-curves $C_1,C_2\subset X$ such that
      $$
         C_1\cdot C_2= 1 \,.
      $$
   \end{itemize}
   There exist K3 surfaces such that (i) and (ii) hold, and
   there exist K3 surfaces where these conditions do not hold.
\end{theorem}

   The equivalence of (i) and (ii) is a consequence of
   Propositions \ref{prop-weil-in-zariski} and
   \ref{prop-zariski-in-weil} below, which
   give criteria for
   mutual inclusions of Zariski and Weyl chambers.
   We thought,
   however, that it might be useful to provide a quick
   direct argument
   for this basic result right away.

\begin{proof}
   Suppose first that condition (ii) does not hold, i.e., that
   there is a pair of $(-2)$-curves $C_1,C_2$
   with $C_1\cdot C_2=1$. Choose an ample divisor $H$ and
   consider the big divisor
   $$
      D=H+a_1C_1+a_2C_2
   $$
   for positive rational numbers $a_i$. Its Zariski decomposition
   is of the form
   $$
      D=(H+x_1C_1+x_2C_2)+(b_1C_1+b_2C_2) \,,
   $$
   where $x_i\le a_i$ and $b_i\le a_i$ are non-negative rational numbers
   satisfying
   $x_i+b_i=a_i$ and
   \begin{equation}\label{determine-Z}
      (H+x_1C_1+x_2C_2)\cdot C_i=0
   \end{equation}
   for $i=1,2$.
   We claim that we can choose the numbers $a_1,a_2$ in such a way that
   $$
      \Neg(D)=\set{C_1,C_2}
      \quad\mbox{and}\quad
      D\cdot C_1>0,\ D\cdot C_2<0 \,.
      \eqno(*)
   $$
   Granting $(*)$ for a moment, we consider the big divisor
   $$
      D' = H+k(C_1+C_2)
   $$
   for $k>0$.
   Thanks to the equality $C_1\cdot C_2=1$ we
   have $D'\cdot C_1<0$ and $D'\cdot C_2<0$ for $k\gg 0$, and hence
   $\Neg(D')=\set{C_1,C_2}$ for $k\gg 0$.
   So we see that $D$ and $D'$ lie
   in the same Zariski chamber, whereas they lie
   in different simple Weyl chambers.

   Turning to the proof of $(*)$, note first
   that the $x_i$ are independent of the $a_i$,
   since they are uniquely determined by the system of linear equations
   \eqnref{determine-Z}, whose coefficient matrix is negative
   definite.
   Keeping this in mind, we see from the equations
   \be
      &&D\cdot C_1=N\cdot C_1=(b_1C_1+b_2C_2)C_1=-2b_1+b_2 \\
      &&D\cdot C_2=N\cdot C_2=(b_1C_1+b_2C_2)C_2=b_1-2b_2
   \ee
   that $(*)$ will be fulfilled if we take
   $$
      a_1=x_1+1 \quad\mbox{and}\quad a_2=x_2+3 \,.
   $$

   Suppose now that condition (ii) holds.
   Let $D$ be a big divisor that does not lie on the boundary of
   any Zariski chamber.
   It is enough to show that if
   $C$ is any $(-2)$-curve, then
   we have $D\cdot C<0$ if $C\in\Neg(D)$, and $D\cdot C>0$
   otherwise.
   Let $D=P+N$ be the Zariski decomposition.
   Then $N=\sum b_iC_i$ with positive rational numbers $b_i$ and
   $(-2)$-curves $C_i$. Thanks to negative definiteness and
   condition (ii), we have $C_i\cdot C_j=0$ for $i\ne j$.
   Therefore if $C=C_1$, say, then
   $$
      D\cdot C=N\cdot C=b_1 C_1^2=-2b_1<0 \,.
   $$
   If $C\not\in\Neg(D)$, then clearly $D\cdot C\ge 0$.
   But in fact we have $D\cdot C>0$, since otherwise
   $C\in\Null(P)$, which by \cite[Prop.~1.7]{BKS}
   would imply that $D$ lies on the boundary
   of a Zariski chamber.

   Finally, we wish to show that there are cases where (i) and
   (ii) hold and cases where these conditions do not hold.
   First,
   there are smooth quartic surfaces in $\P^3$ that contain a
   pair of intersecting lines (see Sect.~\ref{sect-example}),
   and hence
   (ii) does not
   hold on such surfaces.
   Secondly, in order to get examples where (ii) does hold, one could
   take
   K3 surfaces of Picard number one -- there
   are no $(-2)$-curves at all on such surfaces, so that (ii) is trivially
   satisfied.
   More substantially, there are examples of K3
   surfaces of Picard number three containing three
   $(-2)$-curves, where no two of them intersect with
   intersection number one
   (see
   Proposition~\ref{prop-K3-coinciding}).
\end{proof}

   We just saw that the Zariski chamber decomposition may differ
   from the decomposition into simple Weyl chambers.
   By contrast, our next result shows
   that, somewhat
   surprisingly, the \textit{number}
   of Zariski chambers always equals the number of simple Weyl
   chambers.

\begin{theorem}\label{thm-bijection}
   Let $X$ be a K3 surface. Then there is a canonical
   bijection between the
   set of Zariski chambers in $\BigCone(X)$ and the set of simple
   Weyl chambers in $\BigCone(X)$.
\end{theorem}

   Note that the number of chambers may well be infinite, as the
   number of $(-2)$-curves may be infinite
   (cf.~\cite[Remark~7.2]{Kov}).

\begin{proof}
   As before,
   denote by $\mathcal R(X)$ the set of \textit{simple roots}, i.e.,
   the set of $(-2)$-curves on $X$. The set of simple Weyl
   chambers on $X$
   is in bijective correspondence with the set
   $$
      \mathcal W(X)=\sset{\tab[c]{subsets $\mathcal S\subset\mathcal R(X)$ such that
         there is a divisor $D\in\BigCone(X)$ \\ with $D\cdot C<0$ if
         $C\in\mathcal S$
         and $D\cdot C>0$ if $C\in\mathcal R(X)\setminus\mathcal
         S$}}\,,
   $$
   whereas, thanks to
   Lemma~\ref{lemma-neg-def}, the set of Zariski chambers on $X$ is in bijective
   correspondence with the set
   \begin{equation}\label{ZX}
      \mathcal Z(X) =\sset{\tab[c]{finite subsets $\mathcal S\subset\mathcal R(X)$ whose
      intersection matrix \\ is negative definite}}\cup\emptyset
      \,.
   \end{equation}
   (In both cases, the empty set $\emptyset$ corresponds to the
   nef chamber.)
   We will show that $\mathcal W(X)=\mathcal Z(X)$.
   To see this, consider first
   a non-empty set $\mathcal S=\set{\liste C1r}\in\mathcal Z(X)$, and fix an
   ample divisor $H$.
   Then for any non-negative rational numbers $a_i$, the divisor
   $$
      D=H+\sum_i a_iC_i
   $$
   is big.
   We claim that we can choose the $a_i$ in such a way that
   \begin{equation}\label{arrange-a-i}
      D\cdot C_i<0 \qquad\mbox{for all $i$.}
   \end{equation}
   Since clearly $D\cdot C>0$ for all curves $C$ different from
   the $C_i$,
   it follows then that the set $\mathcal S$ is contained
   in $\mathcal W(X)$. To prove \eqnref{arrange-a-i}, consider the
   following system of linear equations for the $a_i$,
   $$
      D\cdot C_j=H\cdot C_j+\sum_{i=1}^r a_iC_i\cdot C_j = -1 \qquad\mbox{for
      $j=1,\dots,r$.}
   $$
   Its coefficient matrix is negative definite and has
   non-negative entries outside of the diagonal. An elementary
   result (see Lemma~\ref{lemma-inverse-matrix})
   implies then that all entries of
   its inverse matrix are $\le 0$. Therefore the solutions
   $a_i$
   are non-negative, and we are done.

   Conversely, consider a non-empty
   set $\mathcal S\in\mathcal W(X)$. The existence of a
   big divisor $D$ with $D\cdot C<0$ for $C\in\mathcal S$ implies
   that $\mathcal S$ is a finite set $\set{\liste
   C1r}$ and that the intersection matrix of $\liste C1r$ is negative
   definite (because the negative part of the Zariski
   decomposition of $D$ must contain the curves $C_i$).
   Therefore $\mathcal S\in\mathcal Z(X)$.
\end{proof}

\section{Inclusions of chambers}\label{sect-inclusions}

   We now give a more detailed description of the mutual
   inclusions of Weyl and Zariski chambers. We continue to use
   the notation $\mathcal Z(X)$ from \eqnref{ZX} for the set that
   consists of all
   sets of $(-2)$-curves whose intersection
   matrix is negative definite. For $\mathcal
   S\in\mathcal Z(X)$, we will write $Z_{\mathcal S}$ for the
   Zariski chamber supported by $\mathcal S$, and $W_{\mathcal
   S}$ for the simple Weyl chamber defined by $\mathcal S$. (In other
   words, $Z_{\mathcal S}$ consists of the big divisors whose
   negative part has support $\bigcup\mathcal S$, and
   $W_{\mathcal S}$ consists of the big divisors that have
   negative intersection with the curves in $\mathcal S$ and
   positive intersection with all other $(-2)$-curves.)

   The following two propositions yield
   Theorem~\ref{introthm-inclusions} from the introduction. As
   mentioned before, this result implies in particular the
   equivalence assertion in Theorem~\ref{introthm-equality-criterion}.

\begin{proposition}\label{prop-weil-in-zariski}
   Let $X$ be a K3 surface, and let $\mathcal S\in\mathcal Z(X)$.
   We have
   $$
      W_{\mathcal S}\subset Z_{\mathcal S}
   $$
   if and only if the following condition holds: If $C'$ is a
   curve with $C'\notin\mathcal S$ such that the intersection
   matrix of the set $\mathcal S\cup\set{C'}$ is negative
   definite, then $C'\cdot C=0$ for all $C\in\mathcal S$.
\end{proposition}

\begin{proof}
   Suppose that the stated condition holds. We need to show that
   one has $\Neg(D)=S$ for every divisor $D\in W_S$. We may
   certainly assume $\mathcal S\neq\emptyset$. Let then $\mathcal
   S =\set{\liste C1r}$ and $D\in W_{\mathcal S}$, so that
   the inequality
   $D\cdot C_i < 0$ holds for $i=1,\dots,r$. In the Zariski
   decomposition $D=P_D+N_D$ we have $P_D\cdot C_i\geq 0$, and
   hence $N_D\cdot C_i<0$ for all $i$. Therefore we obtain in any
   event the inclusion $\mathcal S\subset\Neg(D)$.

   To see the converse inclusion $\Neg(D)\subset\mathcal S$,
   assume by way of contradiction that there is a curve
   $C'\in\Neg(D)$ with $C'\notin\mathcal S$. From our condition
   we get $C_i\cdot C'=0$ for all $i=1,\dots,r$. Let now
   $\bigcup_{j=1}^{s}C'_j$ be the connected component in
   $\Neg(D)$ containing $C'$. Then none of the curves $C'_j$ can
   be contained in $S$, and therefore
   $$
      0\le D\cdot C'_j = N_D\cdot C'_j = \sum_{i=1}^{s}a'_iC'_i\cdot C'_j
   $$
   for $j=1,\dots,s$. Then Lemma~\ref{lemma-inverse-matrix} implies
   that all $a'_i$ must be $\le 0$. But this is a contradiction
   with the fact that $N_D$ is effective. Thus $\Neg(D)=\mathcal
   S$, which means that $D\in Z_{\mathcal S}$.

   To proof the other direction we show: If there is a curve
   $C'\notin\mathcal S$ such that $\mathcal S \cup \set{C'}$ has
   negative definite intersection matrix and if there is a curve
   $C\in\mathcal S$ such that $C\cdot C'=1$, then there is a
   divisor $D$ with $D\in W_{\mathcal S}$ but $D\notin
   Z_{\mathcal{S}}$. Assume that $C'$ is such a curve. Then there
   is by Lemma \ref{lemma-neg-def} a big divisor $B$ with
   $\Neg(B)=\mathcal S \cup \set{C'}$. We can write its Zariski
   decomposition as
   $$
      B=P_B+N_B=P_B + b'C' +\sum_{i=1}^{r}b_i C_i \,,
   $$
   where $b'$ and $\liste b1r$ are positive rational numbers.
   We claim now that we can find positive rational numbers
   $c'$ and $\liste c1r$ such that the divisor
   $$
      D=P_B + c'C' +\sum_{i=1}^{r}c_i C_i
   $$
   satisfies
   \begin{equation}\label{desired-inequ}
      D\cdot C'>0 \quad\mbox{and}\quad
      D\cdot C_i<0 \mbox{ for } i =1,\dots,r
      \,.
   \end{equation}
   These inequalities then tell us that
   $D\in W_S$, but $D\not\in Z_S$.

   To prove the existence of $D$, we solve first
   the system of inequalities
   \begin{eqnarray*}
      D \cdot C' = N_D \cdot C' &<& 0 \\
      D \cdot C_1 = N_D \cdot C_1 &<& 0 \\
      \vdots                      & &   \\
      D \cdot C_r = N_D \cdot C_r &<& 0 \,.
   \end{eqnarray*}
   for the variables $c',\liste c1r$.
   Lemma~\ref{lemma-inverse-matrix} guarantees that the solutions are positive.
   We claim finally that upon
   replacing $c'$ with $c'=\frac14\min c_i$,
   the desired inequalities \eqnref{desired-inequ} hold.
   This latter fact follows from a case-by-case analysis using the
   fact that the connected components of $\Neg(D)$ are A-D-E
   curves (see Figure~\ref{fig-Dynkin} for the possible
   configurations of their components). We omit the details.
\end{proof}

\begin{proposition}\label{prop-zariski-in-weil}
   Let $X$ be a K3 surface, and let $\mathcal S\in\mathcal Z(X)$.
   We have
   $$
      \interior Z_{\mathcal S}\subset W_{\mathcal S}
   $$
   if and only
   if $C_1\cdot C_2\ne 1$ for all curves $C_1,C_2\subset S$.
\end{proposition}

\begin{proof}
   We show first that
   if there are two curves $C_i,C_j\in\mathcal S$ with
   $C_i\cdot C_j=1$, then
   there is a big divisor $E$, which belongs to
   $\interior{Z}_{\mathcal S}$ but not to $W_{\mathcal S}$.
   Setting $\mathcal S=\set{\liste C1r}$, assume that $C_1$ and
   $C_2$ are such curves.
   Choose an ample divisor $H$, and consider the big divisor
   $$
      D=H+\sum_{i=1}^{r} a_iC_i \,,
   $$
   where $\liste a1r$ are positive rational numbers. Its Zariski
   decomposition is of the form
   $$
      D=\underbrace{H+\sum_{i=1}^{r}a^*_iC_i}_{P_D}
      +\underbrace{\sum_{i=1}^{r}\left(a_i-a^*_i\right)C_i}_{N_D}\,,
   $$
   where the coefficients $a^*_i$ satisfy
   $$
      0=P_D\cdot C_i=H\cdot C_i +\sum_{j=1}^{r} a^*_jC_j\cdot
      C_i \qquad\mbox{for}\ i=1,\dots,r.
   $$
   Choosing the $a_i$ large enough,
   we may assume $0<a^*_i<a_i$ (see \cite{Bau}).
   Consider the divisor
   $$
      D^*=H+a^*_1C_1+\sum_{i=2}^{r}a_iC_i
      = P_D+\sum_{i=2}^r (a_i-a_i^*)C_i
      \,.
   $$
   Then we have $P_{D^*}=P_D$, so $\Null(P_{D^*})=\Null(P_D)$, but
   $$
      \Neg(D^*)=\set{\liste C2r}\subsetneq\set{\liste C1r}=\Neg(D)
      \,.
   $$
   Therefore, by \cite[Prop.~1.7]{BKS},
   the divisor $D^*$ lies on the boundary of the
   Zariski chamber $\Sigma_{P_D}$. We have
   $$
      D^*\cdot C_1=P_D\cdot C_1+\sum_{i=2}^r(a_i-a_i^*)C_i\cdot C_1
      \ge(a_2-a_2^*)C_2\cdot C_1=a_2-a_2^*>0 \,.
   $$
   There exists an $\epsilon>0$, such that
   $\left(D^*+\epsilon C_1\right)\cdot C_1>0$, and the divisor
   $D^*+\epsilon C_1$ lies in the interior of $\Sigma_{P_D}$
   by \cite[Prop.~1.8]{BKS}.
   Setting $E=D^{*}+\epsilon C_1$ we are done.

   Conversely, let $D\in\interior{Z}_{\mathcal{S}}$ be a
   big divisor with Zariski decomposition  $D=P_D+N_D$. Then
   $N_D=\sum_{i=1}^{r} a_iC_i$ with positive rational numbers
   $a_i$ and $(-2)$-curves $C_i$. Thanks to negative definiteness
   and the condition in the proposition we have
   $C_i\cdot C_j=0$ for $i\neq j$. Therefore:
   $$
      D\cdot C_i=N_D\cdot C_i=a_i\,C_i^2=-2a_i<0
      \qquad \mbox{for } i=1,\dots,r.
   $$
   For any $(-2)$-curve $C$ different from the $C_i$, we clearly
   have $D\cdot C\ge 0$. We claim that $D\cdot C>0$. In
   fact, if we had $D\cdot C=0$, then
   $P_D\cdot C=0$. So $C\in\Null(P)$, but
   $C\not\in\Neg(D)$, and then, again by
   \cite[Prop.~1.7]{BKS}, $D$ would lie on the boundary of the
   chamber $Z_S$.
\end{proof}

\begin{figure}[t]
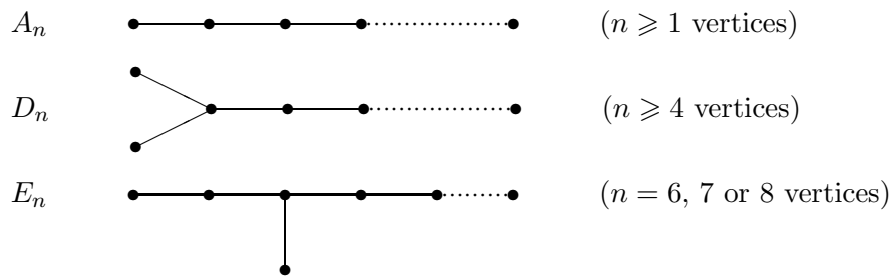

   \begin{flushleft}
      \leftskip=0.1\textwidth
      \ADEdiagrams

      \vspace{2\baselineskip}
      \caption{\label{fig-Dynkin}
         The Dynkin diagrams corresponding to A-D-E curves on a K3
         surface.}
   \end{flushleft}
\end{figure}

\section{Examples}\label{sect-example}

\paragraph{K3 surfaces with differing decompositions.}

   We analyze now in detail
   examples of K3 surfaces where the
   two chamber decompositions differ.
   Specifically, we consider smooth quartic surfaces in $\P^3$
   containing a
   hyperplane section that decomposes into two lines and an
   irreducible conic. It is easy to see -- for instance as in
   \cite[Lemma~2.2b]{Bau97a} -- that such surfaces exist.
   Fix such a quartic surface $X$, and let $L_1,L_2,C$ be two
   lines and an irreducible conic such that $L_1+L_2+C$ is a
   hyperplane section of $X$.
   In the N\'eron-Severi vector space
   $\NS_\R(X)$ we consider the subspace
   $$
      V \eqdef\vecspan{L_1,L_2,C}\subset\NS_\R(X)
   $$
   spanned by the classes of
   $L_1$, $L_2$, and $C$.
   The intersection form is given on $V$ by the matrix
   \begin{equation}\label{L1L2C-matrix}
      \matr{-2 & 1 & 2 \\
         1 & -2 & 2 \\
         2 & 2 & -2}
         \,.
   \end{equation}
   We now show:

\begin{proposition}\label{prop-L1L2C-surface}
   \begin{itemize}\compact
   \item[(i)]
      The intersection $\mathcal C\eqdef\BigCone(X)\cap V$ is
      the interior of the cone generated by the classes of
      $L_1$, $L_2$, and $C$.
   \item[(ii)]
      The curves $L_1$, $L_2$, and $C$ are the only $(-2)$-curves
      in $V$.
   \item[(iii)]
      The intersection $\Nef(X)\cap V$
      consists of the classes $aC+b_1L_1+b_2L_2$ such that the
      real numbers $a,b_1,b_2$ satisfy the inequalities
      $$
         b_1+b_2\ge a,\quad
         2a+b_2\ge 2b_1,\quad
         2a+b_1\ge 2b_2\,.
      $$

   \item[(iv)]
      The cone $\mathcal C$
      decomposes into
      five Zariski chambers and into five simple Weyl
      chambers. These two decompositions do not coincide.
   \end{itemize}
\end{proposition}

\begin{figure}
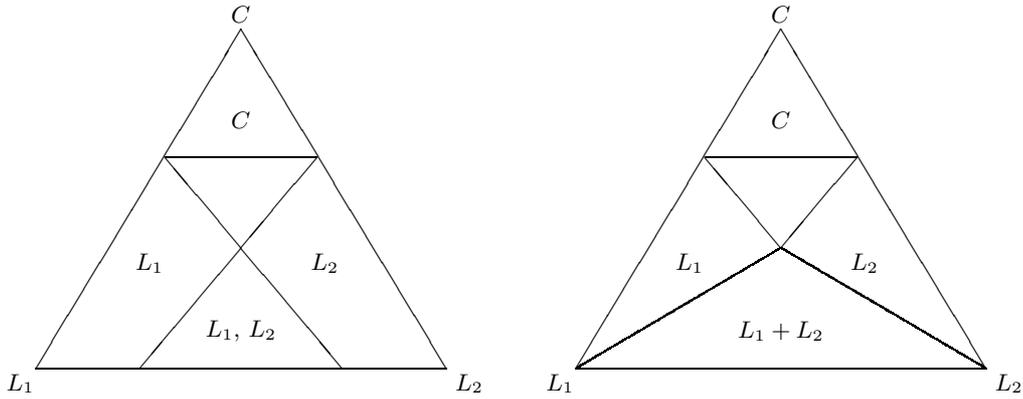

   \begin{center}\footnotesize
      \FigureQuarticChambers
      \caption{\label{fig-quartic-chambers}
         The Weyl chamber decomposition (left picture) and the
         Zariski chamber decomposition (right picture) of the
         cone $\mathcal C$ on the quartic surfaces from
         Proposition~\ref{prop-L1L2C-surface}. The pictures show
         the intersection of the (3-dimensional) cone~$\mathcal
         C$ with the hyperplane in $V$ passing through the
         classes of $L_1$, $L_2$, and $C$. The labels in the
         chambers indicate in the left picture the curves that
         are intersected negatively and in the right picture the
         support of the Zariski chambers. The unlabelled chamber
         in the center is the nef chamber, i.e., the
         chamber whose interior is the ample cone and whose
         closure is the nef cone.
         }
   \end{center}
\end{figure}

\begin{proof}
   (i) As $\BigCone(X)$ is the closure of the effective cone, the
   issue is to see that if a numerical class $D=aC+b_1L_1+b_2L_2\in
   V$ is effective, then all coefficients $a,b_1,b_2$ are
   non-negative. Note that if we had $V=\NS_\R(X)$, then this
   would follow from
   \cite[Theorem~6.1]{Kov}. In our situation it can be seen as
   follows:
   One checks first that the divisor $C+2L_1+2L_2$ is
   nef. From the fact that its intersection with $D$ must be
   non-negative we get then that $a\ge 0$.
   Proceeding in the same way with the nef divisors $C+L_1$ and
   $C+L_2$ we get $b_2\ge 0$ and $b_1\ge 0$ respectively.

   (ii) Suppose that $F$ is a $(-2)$-curve in $V$.
   Then by (i) the numerical
   class of $F$ lies in the closure of $\mathcal C$.
   Writing
   $F=aC+b_1L_1+b_2L_2$ for this class,
   we have
   $$
      (aC+b_1L_1+b_2L_2)F=F^2=-2<0 \,.
   $$
   As $F$ is irreducible and the divisor in brackets
   is effective, this can only happen if
   $F$ is one of the curves $C,L_1,L_2$.

   (iii)
   A class $aC+b_1L_1+b_2L_2$ is nef if and only if it meets the
   curves $C,L_1,L_2$ non-negatively. These three
   conditions yield the
   asserted inequalities.

   (iv)
   The intersection matrix \eqnref{L1L2C-matrix}
   has
   exactly four negative definite prinicipal submatrices,
   corresponding to the four
   divisors
   $$
      L_1,\ L_2,\ L_1+L_2,\ C \,.
   $$
   Therefore,
   by Lemma~\ref{lemma-neg-def},
   there are exactly five Zariski chambers supported on the
   curves $L_1,L_2,C$, and hence
   by Theorem~\ref{thm-bijection}
   also exactly five Weyl chambers.
   The fact that the decompositions do not coincide is of course
   a consequence of Theorem~\ref{thm-equality-criterion}.
   Alternatively, it can be verified directly by computing the
   chambers from the shape of $\Nef(X)\cap V$, i.e., using the
   inequalities given in (iii).
\end{proof}

\begin{remark}\rm
   The two decompositions of $\mathcal C$ are shown in
   Figure~\ref{fig-quartic-chambers}. Using
   Theorem~\ref{introthm-inclusions} we get also information
   about the chamber decomposition of the cone $\BigCone(X)$
   (which might be bigger than $\mathcal C$). In particular, we
   have the following inclusions and non-inclusions of chambers
   on $X$:
   \begin{itemize}
   \item[(i)]
      $\interior Z_\emptyset=W_\emptyset$, as always,
   \item[(ii)]
      $\interior Z_{\set{L_1}}\subsetneq W_{\set{L_1}}$
      and
      $\interior Z_{\set{L_2}}\subsetneq W_{\set{L_2}}$,
   \item[(iii)]
      $\interior Z_{\set{C}}\subset W_{\set{C}}$,
   \item[(iv)]
      $\interior Z_{\set{C}}=W_{\set{C}}$,
      if and only if there is no $(-2)$-curve $C'\subset X$ with
      $C'\cdot C=1$.
   \item[(v)]
      $\interior Z_{\set{L_1,L_2}}\not\subset W_{\set{L_1,L_2}}$
   \item[(vi)]
      $W_{\set{L_1,L_2}}\subset Z_{\set{L_1,L_2}}$, if
      any only if there is no $(-2)$-curve $C'\subset X$ such
      that $C'\cdot L_1=0$ and $C'\cdot L_2=1$ or conversely.
   \end{itemize}
\end{remark}

   The situation becomes particulary transparent when
   one has quartics as above whose Picard number is exactly
   three.
   We will show that
   that surfaces with this property exist:

\begin{proposition}\label{prop-thequartic}
   There exist smooth quartic surfaces $X\subset\P^3$
   of Picard number three
   having a hyperplane section of the form $L_1+L_2+C$ such that
   $L_1$ and $L_2$ are lines and $C$ is a smooth conic.
\end{proposition}

   On such surfaces one has
   $V=\NS_\R(X)$, and then the
   conditions in (iv) and (vi) are fulfilled by
   Proposition~\ref{prop-L1L2C-surface}. Therefore, in that case
   \begin{itemize}
   \item
      $\interior Z_{\set{L_1}}\subsetneq W_{\set{L_1}}$
      and
      $\interior Z_{\set{L_2}}\subsetneq W_{\set{L_2}}$,
   \item
      $\interior Z_{\set{C}}= W_{\set{C}}$,
   \item
      $W_{\set{L_1,L_2}}\subsetneq Z_{\set{L_1,L_2}}$.
   \end{itemize}
   The picture in Figure~\ref{fig-quartic-chambers}
   describes then the whole big cone of $X$.

\begin{proof}[Proof of Proposition~\ref{prop-thequartic}]
   Consider the \textit{K3 lattice}
   $$
      \Lambda_{K_3}=U^3\oplus(-E_8)^2
   $$
   where $U$ is the lattice $\Z^2$ with the bilinear form
   given by the matrix
   $$
      \matr{0 & 1\\ 1 & 0}
   $$
   and $-E_8$ is the lattice $\Z^8$ with the bilinear form
   given by
   $$
      \matr{
         -2&1&0&0&0&0&0&0 \\
         1&-2&1&0&0&0&0&0 \\
         0&1&-2&1&0&0&0&0 \\
         0&0&1&-2&1&0&0&0 \\
         0&0&0&1&-2&1&0&1 \\
         0&0&0&0&1&-2&1&0 \\
         0&0&0&0&0&1&-2&0 \\
         0&0&0&0&1&0&0&-2
         }
   $$
   For any K3 surface $X$, there is an isomorphism
   $\sigma:H^2(X,\Z)\to\Lambda_{K3}$.
   The 21-dimensional \textit{period space}
   $$
      \Omega=\sset{\C\cdot x\in\P(\Lambda_{K_3}\tensor\C)\with\bilin xx=0,\
         \bilin x{\bar x}>0}
   $$
   is a fine moduli space for marked K3 surfaces, i.e., for pairs
   $(X,\sigma)$ consisting of a K3 surface $X$ and an
   isormorphism
   $\sigma:H^2(X,\Z)\to\Lambda_{K3}$
   (see e.g.\ \cite{BPV} for the facts mentioned here).

   Let now $X_0\subset\P^3$ be a smooth quartic surface
   admitting a hyperplane section of the form $L_1+L_2+C$
   as considered above, and let
   $\lambda_1,\lambda_2,\gamma\in\Lambda_{K3}$ be the lattice
   vectors corresponding to the classes of
   $L_1,L_2,C$ under a fixed isomorphism
   $H^2(X,\Z)\isomto\Lambda_{K3}$. The point $x_0\in\Omega$
   corresponding to $X_0$ is contained in the intersection
   $$
      \Omega'=\Omega\cap\lambda_1^\perp\cap\lambda_2^\perp\cap\gamma^\perp
      \ .
   $$
   Consider now a small deformation $X$ of $X_0$
   corresponding to
   a point $x\in\Omega'$ close to $x_0$.
   As $\Lambda_{K3}$ is a countable set,
   the generic such $x$ is not contained in any intersection
   $$
      \Omega'\cap\alpha^\perp
   $$
   where $\alpha\in\Lambda_{K_3}$ is a lattice vector different
   from $\lambda_1,\lambda_2,\gamma$. But this implies that the
   Picard group of $X$ is generated over $\Q$ precisely by the
   classes of $L_1,L_2$ and $C$. The Riemann-Roch theorem implies
   that these $(-2)$-classes are on the deformed surface still
   represented by effective divisors. Therefore, when
   $x$ is close enough to $x_0$, they are still represented by
   irreducible curves.
\end{proof}

\paragraph{K3 surfaces with coinciding decompositions.}

   If a K3 surface does not contain any $(-2)$-curves at all --
   for instance if its Picard group is one-dimensional --
   then of course the Zariski chamber decomposition coincides
   with the Weyl chamber decomposition. More substantial examples
   are given by the following proposition. The double
   covering construction in its proof was
   suggested to us by T.~Szemberg.

\begin{proposition}\label{prop-K3-coinciding}
   There exist K3 surfaces $X$ of Picard number three that contain
   three $(-2)$-curves $F_1,F_2,C$ such that
   $$
      F_1\cdot F_2=0,\quad F_1\cdot C=F_2\cdot C=2 \,,
   $$
   and such that there are no other $(-2)$ curves on $X$.

   The decomposition of the big cone of $X$ into Zariski chambers
   is the same as the decomposition
   into simple Weyl chambers. It consists of five chambers.
\end{proposition}

\begin{proof}
   Let $B\subset\P^2$ be a sextic curve that has two ordinary
   double points $x_1$ and $x_2$
   and no other singularities, and consider the
   blow-up $f:Y=\Bl_{x_1,x_2}(\P^2)\to\P^2$ at these two points. On
   the blow-up let $B'$ be the proper transform of $B$, i.e.,
   $$
      B'=f^*B-2E_1-2E_2 \,,
   $$
   where $E_1$ and $E_1$ are the exceptional divisors over $x_1$
   and $x_2$ respectively. The line bundle $\O_Y(B')$ is then
   2-divisible, so that we can consider the double covering
   $$
      g:X\to Y
   $$
   that is ramified over $B'$.
   As $K_{X}=g^*(K_Y+\frac12 B')=0$ and
   $h^1(\O_X)=0$ (cf.~\cite[Sect.~22]{BPV}),
   the surface $X$ is a K3 surface.
   One checks that the pullbacks
   $$
      F_1=g^*E_1,\quad
      F_2=g^*E_2,\quad
      C=g^*(H-E_1-E_2)
   $$
   are $(-2)$-curves on $X$.
   These curves generate a subspace of the
   N\'eron-Severi vector space of dimension three, on which the
   intersection form is given by the matrix
   \begin{equation}\label{F1F2C-matrix}
      \matr{
         -2 & 0 & 2 \\
         0 & -2 & 2 \\
         2 & 2 & -2
         }
   \end{equation}
   Arguing as in the proof of Proposition~\ref{prop-thequartic},
   i.e., replacing $X$ with a small deformation in the
   period space, we obtain a K3 surface of Picard number three
   with curves $F_1,F_2,C$ intersecting as in
   \eqnref{F1F2C-matrix}
   that generate the N\'eron-Severi group
   over $\Q$.
   For simplicity, let us
   denote the deformed surface again by $X$.

   We show next that there are no other $(-2)$-curves on $X$.
   Suppose that $D$ is a $(-2)$-curve on $X$, and write in
   numerical equivalence
   $$
      D=aC+b_1F_1+b_2F_2
   $$
   with rational numbers $a,b_1,b_2$.
   We can now proceed as in the proof of
   Proposition~\ref{prop-L1L2C-surface}(i)
   to show that $a\ge 0$ and
   $b_1,b_2\ge 0$. (In fact, intersect $D$ with the nef divisors
   $C+F_1$, $C+F_2$, and $C+F_1+F_2$.)
   One sees then as in the proof of
   Proposition~\ref{prop-L1L2C-surface}(ii) that $F_1$, $F_2$,
   and $C$ are
   the only $(-2)$-curves on $X$.
   The intersection matrix
   \eqnref{F1F2C-matrix}
   has exactly four negative definite principal submatrices,
   corresponding to the divisors
   $$
      F_1,\ F_2,\ F_1+F_2,\ C \ .
   $$
   With Lemma~\ref{lemma-neg-def} we conclude that
   there are exactly five Zariski chambers, and
   Theorem~\ref{thm-equality-criterion} implies that the chamber
   decompositions coincide.
\end{proof}

\section*{Appendix}
\renewcommand\thesection{A}
\setcounter{satz}{0} % Damit das Lemma richtig numeriert wird.

   The following useful
   lemma from \cite{BKS} is used
   several times in
   the present paper. While probably well-known, a
   somewhat technical proof was included in \cite{BKS} for lack of a
   reference. Here we give a much simpler argument, which is
   inspired by the beginning of the proof of
   \cite[Theorem~3.2]{PanRos}.

\begin{lemma}\label{lemma-inverse-matrix}
   Let $S$ be a negative definite $(r\times r)$-matrix over the
   reals such that $s_{i,j}\ge 0$ for $i\ne j$. Then all
   entries of the inverse matrix $S\inverse$ are $\le 0$.
\end{lemma}

\begin{proof}
   It is enough to show that if $a,b\in\R^r$ are vectors with
   $a=S\inverse b$ such that
   $b_i\le0$ for all $i$, then $a_i\ge 0$ for all $i$.
   To prove this claim, write $a=p-q$, where $p_i=\max(a_i,0)$.
   Then $p_i\ge 0$, $q_i\ge 0$, and $p_iq_i=0$ for all $i$.
   Using these relations along with
   the hypothesis that $s_{i,j}\ge 0$ for $i\ne j$,
   we get
   $$
      p^tSq=\sum_{i,j}p_is_{i,j}q_j=\sum_{i\ne j}p_is_{i,j}q_j\ge 0
   $$
   and hence
   $$
      p^tSq-q^tSq>0 \qquad\mbox{if }q\ne 0\,,
   $$
   because $S$ is negative definite.
   On the other hand, we have
   $$
      p^tSq-q^tSq=(p-q)^tSq=a^tSq=(Sa)^tq=b^tq\le 0 \,,
   $$
   so we conclude that $q=0$.
   This implies that $a_i\ge 0$ for all $i$, as
   claimed.
\end{proof}

\bigskip
\small
   Tho\-mas Bau\-er,
   Fach\-be\-reich Ma\-the\-ma\-tik und In\-for\-ma\-tik,
   Philipps-Uni\-ver\-si\-t\"at Mar\-burg,
   Hans-Meer\-wein-Stra{\ss}e,
   D-35032~Mar\-burg, Germany.

\nopagebreak
   \textit{E-mail address:} \texttt{tbauer@mathematik.uni-marburg.de}

\bigskip
   Michael Funke,
   Fach\-be\-reich Ma\-the\-ma\-tik und In\-for\-ma\-tik,
   Philipps-Uni\-ver\-si\-t\"at Mar\-burg,
   Hans-Meer\-wein-Stra{\ss}e,
   D-35032~Mar\-burg, Germany.

\nopagebreak
   \textit{E-mail address:} \texttt{funke@mathematik.uni-marburg.de}

\end{document}